\documentclass[10pt,oneside]{amsart}

\usepackage[
paper=a4paper,
headsep=20pt,text={136mm,208mm},centering,includehead
]{geometry}

\usepackage[bookmarks]{hyperref}
\usepackage{amssymb,amsxtra,dsfont}
\usepackage[all]{xy}
\usepackage{pb-diagram,pb-xy}
\usepackage{graphicx,color}
\usepackage{epstopdf}
\usepackage{pinlabel}
\usepackage{float}
\usepackage{array,calc,booktabs}
\usepackage[
  textwidth=1.2in,
  backgroundcolor=yellow,
  bordercolor=orange,
  textsize=small
]{todonotes}
\usepackage{tikz}

\iftrue
\makeatletter
\def\@maketitle{%
  \normalfont\normalsize
  \@adminfootnotes
  \@mkboth{\@nx\shortauthors}{\@nx\shorttitle}%
  \global\topskip42\p@\relax 
  \@settitle
  \ifx\@empty\authors \else \@setauthors \fi
  \ifx\@empty\@dedicatory
  \else
    \baselineskip22\p@
    \vtop{{\small\itshape\@dedicatory\@@par}%
      \global\dimen@i\prevdepth}\prevdepth\dimen@i
  \fi
  \@setabstract
  \normalsize
  \if@titlepage
    \newpage
  \else
    \dimen@25\p@ \advance\dimen@-\baselineskip
    \vskip\dimen@\relax
  \fi
} 
\def\@settitle{%
  \vspace*{-15pt}
  \begin{flushleft}%
    \LARGE\bfseries
    \strut\@title\strut
  \end{flushleft}%
}
\def\@setauthors{%
  \begingroup
  \def\thanks{\protect\thanks@warning}%
  \trivlist
  \raggedright
  \large \@topsep27\p@\relax
  \advance\@topsep by -\baselineskip
  \item\relax
  \author@andify\authors
  \def\\{\protect\linebreak}%
  \authors
  \ifx\@empty\contribs
  \else
    ,\penalty-3 \space \@setcontribs
    \@closetoccontribs
  \fi
  \normalfont
  \endtrivlist
  \endgroup
}
\def\@setaddresses{\par
  \nobreak \begingroup
  \small\raggedright
  \def\author##1{\nobreak\addvspace\smallskipamount}%
  \def\\{\unskip, \ignorespaces}%
  \interlinepenalty\@M
  \def\address##1##2{\begingroup
    \par\addvspace\bigskipamount\noindent
    \@ifnotempty{##1}{(\ignorespaces##1\unskip) }%
    {\ignorespaces##2}\par\endgroup}%
  \def\curraddr##1##2{\begingroup
    \@ifnotempty{##2}{\nobreak\noindent\curraddrname
      \@ifnotempty{##1}{, \ignorespaces##1\unskip}\/:\space
      ##2\par}\endgroup}%
  \def\email##1##2{\begingroup
    \@ifnotempty{##2}{\nobreak\noindent E-mail address%
      \@ifnotempty{##1}{, \ignorespaces##1\unskip}\/:\space
      \ttfamily##2\par}\endgroup}%
  \def\urladdr##1##2{\begingroup
    \def~{\char`\~}%
    \@ifnotempty{##2}{\nobreak\noindent\urladdrname
      \@ifnotempty{##1}{, \ignorespaces##1\unskip}\/:\space
      \ttfamily##2\par}\endgroup}%
  \addresses
  \endgroup
  \global\let\addresses=\@empty
}
\def\@setabstracta{%
    \ifvoid\abstractbox
  \else
    \skip@15pt \advance\skip@-\lastskip
    \advance\skip@-\baselineskip \vskip\skip@
    \box\abstractbox
    \prevdepth\z@ 
    \vskip-15pt
  \fi
}
\renewenvironment{abstract}{%
  \ifx\maketitle\relax
    \ClassWarning{\@classname}{Abstract should precede
      \protect\maketitle\space in AMS document classes; reported}%
  \fi
  \global\setbox\abstractbox=\vtop \bgroup
    \normalfont\small
    \list{}{\labelwidth\z@
      \leftmargin0pc \rightmargin\leftmargin
      \listparindent\normalparindent \itemindent\z@
      \parsep\z@ \@plus\p@
      
    }%
    \item[\hskip\labelsep\bfseries\abstractname.]%
}{%
  \endlist\egroup
  \ifx\@setabstract\relax \@setabstracta \fi
}

\def\ps@headings{\ps@empty
  \def\@evenhead{%
    \setTrue{runhead}%
    \normalfont\scriptsize
    \rlap{\thepage}\hfill
    \def\thanks{\protect\thanks@warning}%
    \leftmark{}{}}%
  \def\@oddhead{%
    \setTrue{runhead}%
    \normalfont\scriptsize
    \def\thanks{\protect\thanks@warning}%
    \rightmark{}{}\hfill \llap{\thepage}}%
  \let\@mkboth\markboth
}\ps@headings

\def\section{\@startsection{section}{1}%
  \z@{-1.4\linespacing\@plus-.5\linespacing}{.8\linespacing}%
  {\normalfont\bfseries\Large}}
\def\subsection{\@startsection{subsection}{2}%
  \z@{-.8\linespacing\@plus-.3\linespacing}{.5\linespacing\@plus.2\linespacing}%
  {\normalfont\bfseries\large}}
\def\subsubsection{\@startsection{subsubsection}{3}%
  \z@{.7\linespacing\@plus.2\linespacing}{-1.5ex}%
  {\normalfont\bfseries}}
\def\@secnumfont{\bfseries}

\renewcommand\contentsnamefont{\bfseries}
\def\@starttoc#1#2{\begingroup
  \setTrue{#1}%
  \par\removelastskip\vskip\z@skip
  \@startsection{}\@M\z@{\linespacing\@plus\linespacing}%
    {.5\linespacing}{
      \contentsnamefont}{#2}%
  \ifx\contentsname#2%
  \else \addcontentsline{toc}{section}{#2}\fi
  \makeatletter
  \@input{\jobname.#1}%
  \if@filesw
    \@xp\newwrite\csname tf@#1\endcsname
    \immediate\@xp\openout\csname tf@#1\endcsname \jobname.#1\relax
  \fi
  \global\@nobreakfalse \endgroup
  \addvspace{32\p@\@plus14\p@}%
  \let\tableofcontents\relax
}
\def\contentsname{Contents}
\def\l@section{\@tocline{2}{.5ex}{0mm}{5pc}{}}
\def\l@subsection{\@tocline{2}{0pt}{2em}{5pc}{}}
\makeatother
\fi 

\def\to{\mathchoice{\longrightarrow}{\rightarrow}{\rightarrow}{\rightarrow}}
\makeatletter
\newcommand{\shortxra}[2][]{\ext@arrow 0359\rightarrowfill@{#1}{#2}}
\def\longrightarrowfill@{\arrowfill@\relbar\relbar\longrightarrow}
\newcommand{\longxra}[2][]{\ext@arrow 0359\longrightarrowfill@{#1}{#2}}

\def\addtagsub#1{\let\oldtf=\tagform@\def\tagform@##1{\oldtf{##1}\hbox{$_{#1}$}}}

\makeatother

\makeatletter
\def\Nopagebreak{\@nobreaktrue\nopagebreak}

\newtheoremstyle{theorem-giventitle}
        {}{}              
        {\itshape}                      
        {}                              
        {\bfseries}                     
        {.}                             
        {\thm@headsep}                             
        {\thmnote{\bfseries#3}}
\newtheoremstyle{theorem-givenlabel}
        {}{}              
        {\itshape}                      
        {}                              
        {\bfseries}                     
        {.}                             
        {\thm@headsep}                             
        {\thmname{#1}~\thmnumber{#3}\setcurrentlabel{#3}}

\newtheoremstyle{definition-giventitle}
        {}{}              
        {}                      
        {}                              
        {\bfseries}                     
        {.}                             
        {\thm@headsep}                             
        {\thmnote{\bfseries#3}}
\def\setcurrentlabel#1{\gdef\@currentlabel{#1}}

\makeatother

\newtheorem{theorem}{Theorem}[section]

\newtheorem{proposition}[theorem]{Proposition}

\newtheorem{lemma}[theorem]{Lemma}

\theoremstyle{definition}
\newtheorem{definition}[theorem]{Definition}

\theoremstyle{theorem-giventitle}
\newtheorem{theorem-named}{}
\theoremstyle{theorem-givenlabel}
\newtheorem{theorem-labeled}{Theorem}

\theoremstyle{definition-giventitle}
\newtheorem{definition-named}{}
\newtheorem{step-named}{}

\numberwithin{equation}{section}

\def\Z{\mathbb{Z}}
\def\Q{\mathbb{Q}}

\def\K{\mathcal{K}}

\def\fs{\mathfrak{s}}
\def\dbar{\bar{d}}
\def\Sq{\Sigma^q}

\def\Ker{\operatorname{Ker}}

\def\Spinc{\operatorname{Spin}^{c}}

\def\rhot{\rho^{(2)}}

\def\setminus{\smallsetminus}

\begin{document}

\vspace*{0mm}

\title[Polynomial splittings]{Polynomial splittings of correction terms and doubly slice knots}

\author{Se-Goo Kim}
\address{Department of Mathematics and Research Institute for Basic
	Sciences, Kyung Hee University, 02447, Korea}
\email{sgkim@khu.ac.kr}
\urladdr{web.khu.ac.kr/\~{}sekim}

\author{Taehee Kim}
\address{
  Department of Mathematics\\
  Konkuk University \\
  Seoul 05029\\
  Korea
}
\email {tkim@konkuk.ac.kr}

\def\subjclassname{\textup{2010} Mathematics Subject Classification}
\expandafter\let\csname subjclassname@1991\endcsname=\subjclassname
\expandafter\let\csname subjclassname@2000\endcsname=\subjclassname


\begin{abstract}
We show that if the connected sum of two knots with coprime Alexander polynomials is doubly slice, then the Ozsv\'{a}th--Szab\'{o} correction terms as smooth double sliceness obstructions vanish for both knots. Recently, Jeffrey Meier gave smoothly slice knots that are topologically doubly slice,  but not smoothly doubly slice. As an application, we give a new example of such knots that is distinct from Meier's knots modulo doubly slice knots. 
\end{abstract}

\maketitle


\section{Introduction}\label{section:introduction}
A knot $K$ in the 3-sphere $S^3$ is \emph{doubly slice} if there is a smoothly embedded and unknotted 2-sphere $S$ in the 4-sphere $S^4$ which transversely intersects the standard $S^3$ in $S^4$ at $K$. If we allow $S$ to be a topologically locally flat embedded and unknotted 2-sphere, then $K$ is called \emph{topologically doubly slice}. A knot is \emph{slice} if it bounds a smoothly embedded disk in the 4-ball $D^4$. Obviously a doubly slice knot is slice.

There have been known results on splittings of sliceness obstructions for knots with coprime Alexander polynomials. Let $K_1$ and $K_2$ be knots and $K:=K_1\# K_2$, their connected sum. Suppose $K_1$ and $K_2$ have coprime Alexander polynomials.
Under this hypothesis, Levine \cite{Levine:1969-2} showed that if $K$ is algebraically slice, then all $K_i$ are algebraically slice, giving $p(t)$-primary decomposition of the algebraic concordance group. The first author \cite{Kim:2005-2} showed that if $K$ has vanishing Casson--Gordon invariants, then so do all $K_i$. The authors \cite{Kim-Kim:2008-1} showed that if $K$ has vanishing metabelian von Neumann--Cheeger--Gromov $\rhot$-invariants, then so do all $K_i$. Later the authors \cite{Kim-Kim:2014-1} extended the result in \cite{Kim-Kim:2008-1} using higher-order $\rhot$-invariants, and gave a knot which has vanishing Casson--Gordon invariants and has concordance genus greater than 1. We note that Cochran--Harvey--Leidy's work in \cite{Cochran-Harvey-Leidy:2009-2} gave evidence of the $p(t)$-primary decomposition of the solvable filtration of Cochran--Orr--Teichner in \cite{Cochran-Orr-Teichner:1999-1} which would extend Levine's $p(t)$-primary decomposition of the algebraic concordance group in \cite{Levine:1969-2}.

Similar results are known for $d$-invariants of Ozsv{\'a}th--Szab\'{o}. For a rational homology 3-sphere $Y$ and $\Spinc(Y)$, the set of $\Spinc$ structures of $Y$, Ozsv{\'a}th--Szab\'{o} \cite{Ozsvath-Szabo:2003-2} defined a function $d\colon \Spinc(Y)\to \Q$, which is called the \emph{correction term}. Via prime power fold branched cyclic covers of $S^3$ over a knot, the $d$-invariants give rise to sliceness obstructions (see Theorem~\ref{theorem:d-invariant as sliceness obstruction}).

For a knot $K$ and a prime power $q=p^r$, we denote by $\Sq(K)$ the $q$-fold branched cover of $S^3$ over $K$. It is known that $H_1(\Sigma^q(K))$ acts freely and transitively on $\Spinc (\Sigma^q(K))$. We denote by $\fs+a$ the element obtained by the action of $a\in H_1(\Sigma^q(K))$ on $\fs\in \Spinc(\Sigma^q(K))$. Let $\fs_0\in\Spinc(\Sigma^q(K))$ be the \emph{canonical $\Spinc$ structure} of $(K,q)$ (see Section~\ref{subsection:spinc structures}), and for each $a\in H_1(\Sq(K))$ define
\[
\dbar(\Sq(K),\fs_0+a):=d(\Sq(K),\fs_0+a)-d(\Sq(K),\fs_0).
\]
We say that a knot $K$ has \emph{vanishing $d$-invariants on $\Sq(K)$} if there exists a metabolizer $G$ for the linking form $H_1(\Sigma^q(K))\times H_1(\Sigma^q(K))\to \Q/\Z$ such that $\dbar(\Sq(K),\fs_0+a)=0$ for all $a\in G$.

Hedden--Livingston--Ruberman \cite[Theorem~3.2]{Hedden-Livingston-Ruberman:2010-1} showed that if $H_1(\Sq(K_1))$ and $H_1(\Sq(K_2))$ have coprime orders and $K_1\# K_2$ is slice, then each $K_i$ has vanishing $d$-invariants on $\Sq(K_i)$. Bao \cite[Theorem~1.1]{Bao:2015-1} showed that if $K_1$ and $K_2$ have coprime Alexander polynomials and $K_1\# K_2$ is slice, then each $K_i$ has vanishing $d$-invariants on $\Sigma^{p^r}\!(K_i)$ for all but finitely many primes $p$ and all natural numbers $r$. She also showed that the finiteness restriction on $p$ can be removed if $K_1\# K_2$ is ribbon.  

We present a new splitting theorem of $d$-invariants concerning double sliceness, and give an application.  

\begin{definition}\label{definition:vanishing d-invariants}
	Given a prime power $q$, a knot $K$ has \emph{doubly vanishing $d$-invariants on $\Sq(K)$} if there exist metabolizers $G_1$ and $G_2$ for the linking form $H_1(\Sigma^q(K))\times H_1(\Sigma^q(K))\to \Q/\Z$ such that $H_1(\Sigma^q(K))=G_1\oplus G_2$ and $\dbar(\Sq(K),\fs_0+a)=0$ for all $a\in G_1\cup G_2$. A knot is said to have \emph{doubly vanishing $d$-invariants} if it has doubly vanishing $d$-invariants on $\Sq(K)$ for every prime power $q$.
\end{definition}
It is well-known to the experts and has appeared in various forms in the literature that a doubly slice knot has doubly vanishing $d$-invariants (for example, see Theorem~\ref{theorem:d-invariant as double sliceness obstruction}). We give the following new splitting theorem of $d$-invariants:
\begin{theorem}[Main Theorem]
\label{theorem:main}
	Let $K_1$ and $K_2$ be knots. Suppose that the Alexander polynomials of $K_1$ and $K_2$ are coprime in $\Q[t^{\pm 1}]$. If $K_1\# K_2$ is doubly slice, then both $K_1$ and $K_2$ have doubly vanishing $d$-invariants.
\end{theorem}

We note that there are knots $K_1$ and $K_2$ such that their Alexander polynomials are coprime, but $H_1(\Sq(K_1))$ and $H_1(\Sq(K_2))$ have the same order for all prime power $q$ \cite{KIm:2009-1}. For these knots the previous theorem is necessary to see that $d$--invariants split. Theorem~\ref{theorem:main} can be considered as an extension of \cite[Theorem~1.1]{Bao:2015-1} to the case of double sliceness, and can be proved using Seifert forms similarly as done by Bao for the case of sliceness. However we prove it using Blanchfield forms, which, we think, is a simpler way.

Unlike Bao's theorem, we do not need any finiteness condition on primes $p$ for double sliceness. A crucial distinction between doubly slice knots and slice knots is that a doubly slice knot bounds two slice disks $D_\pm$ such that $D_+\cup D_-$ is an unknotted 2-sphere in $S^4$, and therefore $H_1(D^4\setminus D_\pm;\Z[t^{\pm }])$ are $\Z$-torsion free. See the proof of Theorem~\ref{theorem:main} for details.  

Meier \cite{Meier:2015-1} gave an infinite family of slice knots $K_p$ (denoted $\K_{p,k}$ in \cite{Meier:2015-1}) for odd primes $p$ that are topologically doubly slice, but not doubly slice.  As an application of our main theorem, we give another example of a slice knot $K$ distinct from all $K_p$, modulo double sliceness, that is  topologically doubly slice, but not doubly slice. Explictly, let $T$ be the positive untwisted Whitehead double of the right-handed trefoil and let $K$ be the knot given in Figure~\ref{figure1}. 
\begin{figure}
\setlength{\unitlength}{0.5pt}
\begin{picture}(258,254)
\put(0,0){\includegraphics[scale=0.5]{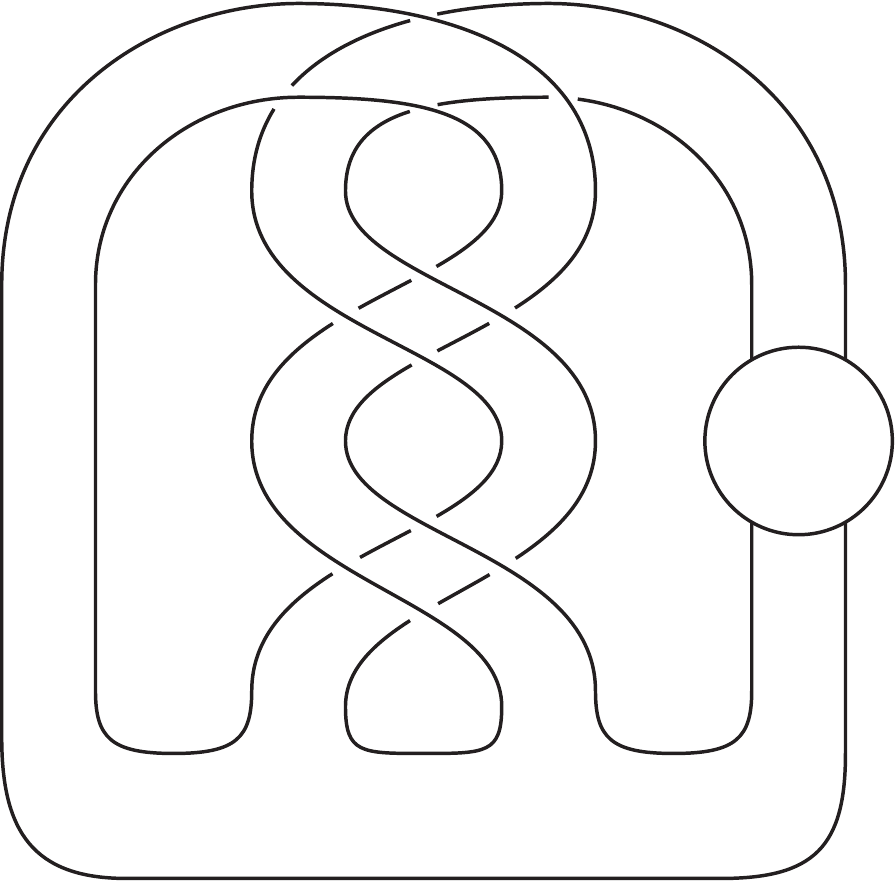}}
\put(232,127){\makebox(0,0){\large $T$}}
\end{picture}
\caption{The knot $K$.}\label{figure1}
\end{figure}
Namely, $K$ is a satellite of $T$ with pattern the $9_{46}$ knot.  We have the following theorem.

\begin{theorem}\label{theorem:example} Let $K$ be the knot in Figure~1. Let $n,n_i\in \Z$ and let $p_i$ be odd primes. 
\begin{enumerate}
	\item The knot $K$ is slice, topologically doubly slice, but not doubly slice. Indeed, $K$ does not have doubly vanishing $d$-invariants..
	\item Suppose $n=1$ or $n_i=1$ for some $i$. Then, the knot $nK\# (\#_{i=1}^m n_iK_{p_i})$ is not doubly slice. In particular, $K\#(-K_{p_i})$ is not doubly slice for all $p_i$.
\end{enumerate}
\end{theorem}

This paper is organized as follows. We give the background material on Blachfield forms, linking forms, and the correction terms in Section~\ref{section:preliminaries}. We prove Theorems~\ref{theorem:main} and \ref{theorem:example} in Section~\ref{section:polynomial splitting}.

\subsection*{Acknowledgments} 
The first author was supported by Basic Science Research Program through the National Research Foundation of Korea(NRF) funded by the Ministry of Education(NRF-2011-0012893).
The second author was supported by Basic Science Research Program through the National Research Foundation of Korea(NRF) funded by the Ministry of Education (no. 2011-0030044(SRC-GAIA) and no. 2015R1D1A1A01056634).

\section{Preliminaries}\label{section:preliminaries}
\subsection{Blanchfield forms and linking forms}\label{subsection:Blanchfield form}
For a knot $K$, let $M(K)$ denote the zero framed surgery on $K$ in $S^3$. Let $\Lambda:=\Z[t^{\pm 1}]$. It is known that $H_1(S^3\setminus K;\Lambda)\cong H_1(M(K);\Lambda)$ as $\Lambda$-modules, and there is a nonsingular hermitian sesquilinear form 
\[
B\ell\colon H_1(M(K);\Lambda)\times H_1(M(K);\Lambda)\to S^{-1}\Lambda/\Lambda,
\]
which is called the \emph{Blanchfield form} of $K$. Here $S^{-1}=\{f\in\Lambda\,\mid\, f(1)=1\}$. For a $\Lambda$-submodule $P$ of $H_1(M(K);\Lambda)$, let 
\[
P^\perp:=\{y\in H_1(M(K);\Lambda)\,\mid\, B\ell(x,y)=0\mbox{ for all }x\in P\}.
\]
We say that $P$ is a \emph{metabolizer} for the Blanchfield form if $P=P^\perp$. 

For a disk $D$ properly embedded in $D^4$,  let $X(D):=D^4\setminus N(D)$ where $N(D)$ is the open tubular neighborhood of $D$. If $D$ is a slice disk for a knot $K$, then $\partial X(D) = M(K)$. We have the following well-known proposition, for example see \cite[Proposition~2.7]{Friedl:2003-6}.

\begin{proposition}\label{proposition:metabolizer for Blanchfield form}
	Let $K$ be a slice knot and $D$ a slice disk for $K$ in $D^4$. Suppose $H_1(X(D);\Lambda)$ is $\Z$-torsion free. Then $\Ker\{i_*\colon H_1(M(K);\Lambda)\to H_1(X(D);\Lambda)\}$is a metabolizer for the Blanchfield form, where $i_*$ is the homomorphism induced from the  inclusion.
\end{proposition}

Recall that for a prime power $q$ we let $\Sq(K)$ denote the $q$-fold branched cyclic cover of $S^3$ over $K$. Then $\Sq(K)$ is a rational homology 3-sphere and there is a nonsingular linking form $\lambda^q\colon H_1(\Sq(K))\times H_1(\Sq(K))\to \Q/\Z$. For a subgroup $P$ of $H_1(\Sq(K))$, we define 
\[
P^\perp:=\{y\in H_1(\Sq(K))\,\mid\, \lambda^q(x,y)=0\mbox{ for all }x\in P\}.
\]

We say that $P$ is a \emph{metabolizer} for $\lambda^q$ if $P=P^\perp$. If $P$ is a metabolizer, we have $|P|^2= |H_1(\Sq(K))|$. It is known that $H_1(\Sq(K))=H_1(M(K);\Lambda)/(t^q-1)$ as $\Lambda$-modules, and we say that $P$ is a \emph{$\Lambda$-metabolizer} if $P=P^\perp$ and $P$ is a $\Lambda$-submodule of $H_1(\Sq(K))$. For a disk $D$ properly embedded in $D^4$, let $W^q(D)$ be the $q$-fold branched cyclic cover of $D^4$ over $D$. If a knot bounds a slice disk $D$ in $D^4$, then $W^q(D)$ is a rational homology 4-ball such that $\partial W^q(D)=\Sq(K)$. We have the following well-known proposition, for example see \cite[Proposition~2.15]{Friedl:2003-6}.

\begin{proposition}\label{proposition:metabolizer for linking form}
Let $K$ be a slice knot and $D$ a slice disk for $K$ in $D^4$. Then $\Ker\{H_1(\Sq(K))\to H_1(W^q(D))\}$ is a $\Lambda$-metabolizer for the linking form $\lambda^q$. 
\end{proposition}

We relate Blanchfield forms to  linking forms.  Define 
\[
\pi^q\colon H_1(M(K);\Lambda)\to H_1(M(K);\Lambda)/(t^q-1)=H_1(\Sq(K))
\] 
to be the projection map. 

\begin{proposition}[{\cite[Proposition~2.18]{Friedl:2003-6}}]
\label{proposition:projection of metabolizers}
Let $K$ be a knot. If $P\subset H_1(M(K);\Lambda)$ is a metabolizer for the Blanchfield form on $H_1(M(K);\Lambda)$, then the submodule $\pi^q(P)\subset H_1(\Sq(K))$ is a $\Lambda$-metabolizer for the linking form $\lambda^q$ on $H_1(\Sq(K))$. 
\end{proposition}

\subsection{$\Spinc$ structures and correction terms}\label{subsection:spinc structures}
Let $\fs_0\in\Spinc(\Sigma^q(K))$ be the \emph{canonical $\Spinc$ structure} of $(K,q)$ defined as follows: let $f\colon \Sigma^q(K)\to S^3$ be the branched covering map and let $K':=f^{-1}(K)$. Then $\fs_0$ is defined to be the unique $\Spinc$ structure whose restriction to $\Sigma^q(K)\setminus N(K')$ is the pull-back $f^*(\fs)$ of the unique $\Spinc$ structure $\fs$ on $S^3\setminus N(K)$. Here $N(K')$ and $N(K)$ denote the open tubular neighborhoods of $K'$ and $K$, respectively. We note that the $\Spinc$ structure $\fs_0$ is equal to the $\Spinc$ structure given in \cite[Lemma~2.1]{Grigsby-Ruberman-Saso:2008-1} (see \cite[Remark~2.5]{Jabuka:2012-1}). For more details, refer to \cite[Section~2]{Jabuka:2012-1}. Now using $\fs_0$ we can identify $H_1(\Sq(K))$ with $\Spinc(\Sq(K))$ via the map $a\mapsto \fs_0 + a$ for $a\in H_1(\Sq(K))$.

With the canonical $\Spinc$ structure $\fs_0\in \Spinc (\Sq(K))$, we have the following sliceness obstruction. 

\begin{theorem}[{\cite{Grigsby-Ruberman-Saso:2008-1}}]
\label{theorem:d-invariant as sliceness obstruction} Let $q$ be a prime power. Suppose that a knot $K$ bounds a slice disk $D$ in $D^4$. Then, $d(\Sq(K),\fs_0+a)=0$ for all $a\in \Ker\{i_*\colon H_1(\Sq(K))\to H_1(W^q(D))\}$ where $i_*$ is the homomorphism induced from the inclusion. In particular, if $K$ is slice, then $d(\Sq(K),\fs_0)=0$ and moreover there is a $\Lambda$-metabolizer $P\subset H_1(\Sq(K))$ such that $d(\Sq(K),\fs_0+a)=0$ for all $a\in P$.
\end{theorem}

As obstructions for a knot to being doubly slice, we have the following theorem.
\begin{theorem}[{\cite[Theorem~2.2]{Meier:2015-1}}]
\label{theorem:d-invariant as double sliceness obstruction} Let $q$ be a prime power. If $K$ is a doubly slice knot, then $H_1(\Sq(K)) = M_+\oplus M_-$ where $M_\pm$ are $\Lambda$-metabolizers for the linking  form $\lambda^q$ on $H_1(\Sq(K))$ and $d(\Sq(K), \fs_0+a)=0$ for all $a\in M_+\cup M_-$. In particular, $K$ has doubly vanishing $d$-invariants.
\end{theorem}

\section{Proofs of Theorems~\ref{theorem:main} and \ref{theorem:example}}
\label{section:polynomial splitting}
First we give a proof of Theorem~\ref{theorem:main}. 

\begin{proof}[Proof of Theorem~\ref{theorem:main}]
Let $q$ be a prime power, and $K:=K_1\# K_2$. Since $K$ is doubly slice, $K=S\cap S^3$ where $S$ is an unknotted 2-sphere in $S^4$ which transversely intersects $S^3$, where $S^3$ is the standard 3-sphere in $S^4$. We regard $S^4=D^4_+\cup D^4_-$, a union of two 4-balls, such that $\partial D^4_{\pm}=S^3$. Let $D_+:=S\cap D^4_+$ and $D_-:=S\cap D^4_-$. 

Since $S$ is unknotted in $S^4$, using a Mayer-Vietoris sequence, we obtain
\[
H_1(S^3\setminus K;\Lambda)=H_1(D^4_+\setminus D_+;\Lambda)\oplus H_1(D^4_-\setminus D_-;\Lambda).
\]
Since $H_1(S^3\setminus K;\Lambda)$ is $\Z$-torsion free, the summands on the right hand side are $\Z$-torsion free. Therefore, letting 
\[
M_\pm :=\Ker\{i_\pm\colon H_1(S^3\setminus K;\Lambda)\to H_1(D^4_\pm\setminus D_\pm;\Lambda)\} 
\]
where $i_\pm$ are homomorphisms induced from inclusions, by Proposition~\ref{proposition:metabolizer for Blanchfield form} the $M_\pm$ are metabolizers for the Blanchfield form 
\[
B\ell\colon H_1(S^3\setminus K;\Lambda)\times H_1(S^3\setminus K;\Lambda) \to S^{-1}\Lambda/\Lambda
\] 
and $H_1(S^3\setminus K;\Lambda)= M_+\oplus M_-$.

Note that $H_1(S^3\setminus K;\Lambda)= H_1(S^3\setminus K_1;\Lambda)\oplus H_1(S^3\setminus K_2;\Lambda)$. For $i=1,2$, let
\[
M^i_\pm:=M_\pm \cap H_1(S^3\setminus K_i;\Lambda).
\]
Then clearly $M^1_\pm \oplus M^2_\pm \subset M_\pm$. We need the following lemma, of which proof will be given later.

\begin{lemma}\label{lemma:splitting of metabolizers}
	$M_\pm = M^1_\pm \oplus M^2_\pm$ and $M^i_\pm$ are metabolizers for the Blanchfield form on $H_1(S^3\setminus K_i;\Lambda)$ for $i=1,2$.
\end{lemma}

Let $W_+$ (resp. $W_-$) be the $q$-fold cyclic cover of $D^4_+$ (resp. $D^4_-$) branched over $D_+$ (resp. $D_-$). Then $\partial W_\pm = \Sq(K)$. Note that 
\begin{align*}
H_1(\Sq(K)) &\cong H_1(S^3\setminus K;\Lambda)/(t^q-1),\\
H_1(W_\pm) &\cong H_1(D^4_\pm \setminus D_\pm;\Lambda)/(t^q-1).
\end{align*}

Therefore we have the following commutative diagram:
\[ \xymatrix{ 
H_1(S^3\setminus K;\Lambda)	\ar[r]^{\,\,\, f_*}\ar[d]^{i_\pm}&
    H_1(\Sq(K))\ar[d]^{j_\pm}\\
    H_1(D^4_\pm \setminus D_\pm;\Lambda) \ar[r]^{\,\,\, g_*}&
    H_1(W_\pm).    }
  \]
In the above diagram, $j_\pm$ are homomorphisms induced from inclusions, and $f_*$ and $g_*$ are the canonical surjections  sending a $\Lambda$-module to its quotient by $t^q-1$.

Let $G_\pm :=f_*(M_\pm)$ and $G^i_\pm:=f_*(M^i_\pm)$ for $i=1,2$. Since we have $H_1(S^3\setminus K;\Lambda)=H_1(S^3\setminus K_1;\Lambda)\oplus H_1(S^3\setminus K_2;\Lambda)$ and $H_1(\Sq(K))=H_1(\Sq(K_1))\oplus H_1(\Sq(K_2))$ as $\Lambda$-modules and $f_*$ preserves the direct sums, we have $G^i_\pm\subset H_1(\Sq(K_i))$ for $i=1,2$. Since $M^i_\pm$ are metabolizers for the Blanchfield form on $H_1(S^3\setminus K_i;\Lambda)$ for $i=1,2$, by Proposition~\ref{proposition:projection of metabolizers}, $G^i_\pm$ are metabolizers for the linking form $H_1(\Sq(K_i)\times H_1(\Sq(K_i))\to \Q/\Z$ for $i=1, 2$. Similarly, $G$ is a metabolizer for the linking form on $H_1(\Sq(K))$.  Moreover, we have $H_1(\Sq(K_i))=G^i_+\oplus G^i_-$ for $i=1,2$. 

Now we show that $\dbar(\Sq(K_1),\fs_0+a)=0$ for all $a\in G^1_+\cup G^1_-$. Let $a\in G^1_+$. Then $(a,0)\in G^1_+\oplus G^2_+ = G_+\subset H_1(\Sq(K))$. Since $j_+(G_+)=(j_+\circ f_*)(M_+)=(g_*\circ i_+)(M_+) = g_*(0) = 0$, we have $G_+\subset \Ker(j_+)$. Furthermore, since both $G_+$ and $\Ker(j_+)$ are metabolizers for the linking form on $H_1(\Sq(K))$, we have $|G_+|=|\Ker(j_+)|$ and hence $G_+=\Ker(j_+)$. Therefore, since $D_+$ is a slice disk for $K$, $d(\Sq(K), \fs_0+(a,0))=0$ by Theorem~\ref{theorem:d-invariant as sliceness obstruction}. Again, since $K$ is slice, by Theorem~\ref{theorem:d-invariant as sliceness obstruction} we have $d(\Sq(K),\fs_0)=0$.  Therefore,
\begin{alignat*}{2}
0	&=	&&d(\Sq(K),\fs_0+(a,0))\\
	&= 	&&d(\Sq(K),\fs_0+(a,0))-d(\Sq(K),\fs_0)\\
	&= 	&&d(\Sq(K_1),\fs_0+a)+d(\Sq(K_2),\fs_0)-\\
	&   	&&(d(\Sq(K_1),\fs_0)	+ d(\Sq(K_2),\fs_0))\\
	&=	&&d(\Sq(K_1),\fs_0+a)) - d(\Sq(K_1),\fs_0)\\
	&=	&&\dbar(\Sq(K_1), \fs_0+a).
\end{alignat*}
Similarly, $\dbar(\Sq(K_1),\fs_0+a)=0$ for all $a\in G^1_-$. Therefore $K_1$ has doubly vanishing $d$-invariants. Using similar arguments, one can show that $K_2$ also has doubly vanishing $d$-invariants.
\end{proof}

Next, we give a proof of Lemma~\ref{lemma:splitting of metabolizers}.
\begin{proof}[Proof of Lemma~\ref{lemma:splitting of metabolizers}]
The proof is almost the same as the one of \cite[Theorem~3.1]{Kim-Kim:2008-1}. The inclusion $M^1_+\oplus M^2_+\subset M_+$ is obvious, and we show that $M_+\subset M^1_+\oplus M^2_+$. For brevity, let $\Delta_i:=\Delta_{K_i}(t)$, the Alexander polynomial of $K_i$, for $i=1,2$. Since $\Delta_1$ and $\Delta_2$ are coprime in $\Q[t^{\pm 1}]$, there exist $\overline{f_1}, \overline{f_2}\in \Q[t^{\pm 1}]$ such that $\overline{f_1}\Delta_1 + \overline{f_2}\Delta_2=1$. Let $c$ be an integer such that $c\overline{f_1}, c\overline{f_2}\in \Lambda$, and let $f_1:=c\overline{f_1}$ and $f_2:=c\overline{f_2}$. Then we have $f_1\Delta_1 + f_2\Delta_2 = c\in \Z$. 

Let $z\in M_+$. Then $z=(x,y)\in H_1(S^3\setminus K_1;\Lambda)\oplus H_1(S^3\setminus K_2;\Lambda)$. Since $\Delta_1$ annihilates $H_1(S^3\setminus K_1;\Lambda)$, we have $\Delta_1x = 0$. Therefore, $cx = (f_1\Delta_1+ f_2\Delta_2)x = f_2\Delta_2x$. Since $\Delta_2$ annihilates $H_1(S^3\setminus K_2;\Lambda)$, we have $\Delta_2y=0$ and hence $f_2\Delta_2y=0$. Therefore $f_2\Delta_2z = (f_2\Delta_2x,0) = (cx,0)$, and we have $(cx,0)\in M_+$. Since $M_+$ is a direct summand of $H_1(S^3\setminus K;\Lambda)$, which is $\Z$-torsion free, we have $(x,0)\in M_+$. Therefore $x\in M^1_+$. Similarly, one can show that $y\in M^2_+$, and now we have $z\in M^1_+\oplus M^2_+$. It follows that $M_+ = M^1_+\oplus M^2_+$. Similarly, one can show that $M_-=M^1_-\oplus M^2_-$. 

Now we show that for each $i=1,2$, the modules $M^i_\pm$ are metabolizers of the Blanchfield form $B\ell_i\colon H_1(S^3\setminus K_i;\Lambda)\times H_1(S^3\setminus K_i;\Lambda)\to S^{-1}\Lambda/\Lambda$. Let $x_1,x_2\in M^1_+$. Then $(x_1,0),(x_2,0)\in M^1_+\oplus M^2_+ = M_+$, which is a metabolizer. Therefore $B\ell((x_1,0),(x_2,0)) = 0$. Since $B\ell((x_1,0),(x_2,0))=B\ell_1(x_1,x_2)+ B\ell_2(0,0) = B\ell_1(x_1,x_2)$, we have $B\ell_1(x_1,x_2)=0$. This implies that $M^1_+\subset (M^1_+)^\perp$.

Conversely, let $x_1\in (M^1_+)^\perp$. Then, for every $(x_2,y)\in M^1_+\oplus M^2_+=M_+$, we have $B\ell((x_1,0),(x_2,y))=B\ell_1(x_1,x_2) + B\ell_2(0,y) = 0 + 0 = 0$. Therefore $(x_1,0)=(M_+)^\perp = M_+$. Since $M_+=M^1_+\oplus M^2_+$, we have $x_1\in M^1_+$. Therefore $(M^1_+)^\perp\subset M^1_+$, and consequently we have $M^1_+=(M^1_+)^\perp$ and $M^1_+$ is a metabolizer for $B\ell_1$. 

Similarly, we can show that $M^2_+$, $M^1_-$, and $M^2_-$ are also metabolizers. 
\end{proof}

We finally prove Theorem~\ref{theorem:example}.
\begin{proof}[Proof of Theorem~\ref{theorem:example}]
We prove Part~(1). The knot $K$ is slice since there is a surgery curve for a slice disk on the left band of the obvious Seifert surface for $K$. Recall that $K$ is a satellite of $T$ with pattern the $9_{46}$ knot. It is known that $9_{46}$ is doubly slice, and since $\Delta_T(t)=1$, $T$ is topologically doubly slice by Freedman's work. Now $K$ is a satellite of a topologically doubly slice knot whose pattern is doubly slice, and therefore $K$ is topologically doubly slice (see \cite[Proposition~3.4]{Meier:2015-1}). 

We show that $K$ is not doubly slice. The needed computation is already done in \cite{Cochran-Harvey-Horn:2012-1}. Our knot $K$ is the same as the knot $K=R(J,T)$ in \cite[Figure~8.1]{Cochran-Harvey-Horn:2012-1} with the choice $J=U$, the unknot. As computed in \cite[Section~8]{Cochran-Harvey-Horn:2012-1}, for the 3-fold branched cyclic cover of $S^3$ over $K$, one can show that $H_1(\Sigma^3(K))\cong \Z_7\langle x_1\rangle \oplus \Z_7\langle y_1\rangle$ and the linking form on $H_1(\Sigma^3(K))$ has only two metabolizers $\langle x_1\rangle$ and $\langle y_1\rangle$. In the proof of Lemma~8.2 in \cite{Cochran-Harvey-Horn:2012-1}, it is assumed that  $K=R(U,T)$, which is the same as our $K$, and it is computed that $d(\Sigma^3(K), \fs_0+4x_2)\le -\frac32$ for the element $x_2$ such that $4x_2=x_1$. Since the linking form on $H_1(\Sigma^3(K))$ has only two metabolizers $\langle x_1\rangle$ and $\langle y_1\rangle$, it follows that $K$ does not have doubly vanishing $d$-invariants. In particular, by Theorem~\ref{theorem:d-invariant as double sliceness obstruction}, $K$ is not doubly slice. 

We prove Part~(2). Since each $K_{p_i}$ has the same Alexander polynomial with $T_{2,p_i}\# T_{2,-p_i}$, we have $\Delta_{K_{p_i}}(t) = \phi_{2p_i}^2$ where $\phi_q$ denotes the $q$-th cyclotomic polynomial. Also note that $\Delta_K(t) = (2t-1)(t-2)$. Therefore all $nK$ and $n_iK_{p_i}$ have mutually coprime Alexander polynomials. 

Suppose $n=1$. By Part~(1), the knot $K$ does not have doubly vanishing $d$-invariants. Therefore by Theorem~\ref{theorem:main} the knot $K\# (\#_{i=1}^m n_iK_{p_i})$ is not doubly slice. 

Suppose $n_i=1$ for some $i$. By rearranging $p_i$, we may assume $n_1=1$. It was shown in \cite[Corollary~5.2]{Meier:2015-1} that $K_{p_i}$ does not have doubly vanishing $d$-invariants. Again by Theorem~\ref{theorem:main} the knot $nK\# (\#_{i=1}^m n_iK_{p_i})=K_{p_1}\# nK\# (\#_{i=2}^m n_iK_{p_i})$ is not doubly slice. 
\end{proof}

\bibliographystyle{amsalpha}
\renewcommand{\MR}[1]{}
\providecommand{\bysame}{\leavevmode\hbox to3em{\hrulefill}\thinspace}
\providecommand{\MR}{\relax\ifhmode\unskip\space\fi MR }
\providecommand{\MRhref}[2]{%
  \href{http://www.ams.org/mathscinet-getitem?mr=#1}{#2}
}
\providecommand{\href}[2]{#2}


\end{document}